\documentclass[11pt]{amsart}

\usepackage{amsmath}
\usepackage{amssymb}
\usepackage{amsthm}
\usepackage{bbm}
\usepackage{mathrsfs}
\usepackage{epsfig}
\usepackage[usenames,dvipsnames]{xcolor}
\usepackage[colorlinks=true,linkcolor=NavyBlue,urlcolor=RoyalBlue,citecolor=PineGreen,bookmarks=true,bookmarksopen=true,bookmarksopenlevel=2,unicode=true,linktocpage]{hyperref}

\newtheorem{theorem}{Theorem}[section]

\newtheorem{corollary}[theorem]{Corollary}

\definecolor{dred}{rgb}{0.8,0.1,0.2}

\definecolor{dblue}{rgb}{0.8,0.2,0.9}

\def\bC{{\sf C}}
\def\C{{\mathbb C}}
\def\Cyc{{\mathscr C}}
\def\dDelta{\Delta}
\def\E{{\sf E}}

\def\epsilon{\varepsilon}
\def\F{{\mathscr F}^{\text{\raisebox{1.5pt}{$\scriptscriptstyle\rightarrow$}}}}
\def\Fe{{\sf F}}
\def\G{{\sf G}}
\def\GbN{\G^{\text{\scalebox{0.8}{\rotatebox{90}{$\bowtie$}}}N}}
\def\geq{\geqslant}

\def\H{{\mathbb H}}
\def\hb{h^{\text{\scalebox{0.8}{\rotatebox{90}{$\bowtie$}}}}}
\def\Hcyc{{\mathscr H}}
\def\HF{{\mathscr{H\! F}}^{\text{\raisebox{1.5pt}{$\scriptscriptstyle\rightarrow$}}}}

\def\R{{\mathbb R}}
\def\red#1{{\color{dred}{#1}}}
\def\S{{\mathfrak S}}
\def\Scyc{{\mathscr S}}
\def\SU{{\rm SU}}
\def\Tr{{\rm Tr}}
\def\U{{\sf U}}
\def\UUF{\mathscr{F}}
\def\V{{\sf V}}
\def\W{{\sf W}}
\def\Z{{\mathbb Z}}

\title{A colourful path to matrix-tree theorems}

\author{Adrien Kassel}
\address{Adrien Kassel -- CNRS, \'Ecole Normale Sup\'erieure de Lyon}
\email{adrien.kassel@ens-lyon.fr}

\author{Thierry L\'evy}
\address{Thierry L\'evy -- LPSM, Sorbonne Universit\'e, Paris}
\email{thierry.levy@sorbonne-universite.fr}

\date{\today}

\thanks{\textit{Acknowledgments of support.} A.K. thanks CNRS for support through a PEPS grant in 2018 and the Newton Institute for hospitality in 2015 during the programme \emph{Random Geometry}. T.L. would like to thank the Isaac Newton Institute for Mathematical Sciences for support and hospitality during the programme {\em Scaling limits, rough paths, quantum field theory} when work on this paper was pursued. This work was supported by EPSRC grant numbers EP/K032208/1 and EP/R014604/1.}



\begin{document}

\maketitle

\begin{abstract} 
In this short note, we revisit Zeilberger's proof of the classical matrix-tree theorem and give a unified concise proof of variants of this theorem, some known and some new.
\end{abstract}

\section{Introduction}

It is well known in combinatorics that determinants of matrices can be interpreted as weighted sums over paths in graphs. One famous and prominent example is the matrix-tree theorem and its variants which interpret minors of a discrete Laplacian as weighted counts of spanning forests. A beautiful proof of its most classical version was given by Doron Zeilberger in~\cite[Section~4]{Zeilberger}, a short paper which emphasizes this combinatorial approach to matrix algebra, a point of view which seems to have been newer and less popular at the time of its publication. 

In the course of our research in probability and stochastic processes on graphs with gauge symmetry~\cite{Kassel-ESAIM,KL1,KL-QSF}, we were led to reconsider this classical theorem, adding in new parameters in the guise of matrices attached to edges of a directed graph. While doing so we revisited Zeilberger's proof, in its memorable colourful re-enactment performed for us by David Wilson some years ago, and found a concise and pictorial way to prove in a unified way several such theorems, notably the classical one of Kirchhoff~\cite{Kirchhoff}, the signed version of Zaslavsky~\cite[Theorem~8A.4]{Zaslavsky}, the abelian-group-weighted version of Chaiken~\cite[Eq. (15) and the following unnumbered equation]{Chaiken}, a special case of which is the complex-weighted one of Forman~\cite[Theorem~1]{Forman}, and the quaternion case of Kenyon~\cite[Theorem~9]{Kenyon}. Let us emphasize that this proof is one of those which does not use the Cauchy--Binet theorem, a step which always blurs somewhat the understanding of cancellations in the signed combinatorial sums involved.

We believe this proof to be crisp and clear enough to be shared with the combinatorics community, and hope it can publicise Zaslavsky's, Chaiken--Forman's, and Kenyon's apparently lesser known theorems. We also believe this approach may inspire the discovery of new formulas and we indeed provide at least a few new identities, which although not definitive, show some insight into the combinatorics of graphs with parameters and the power of Zeilberger's approach.

The note is organised as follows. We first recall definitions about determinants, then proceed to give our revisited and unified proof. Last, we prove new results and discuss in an open-ended way some natural perspectives.

\section{A unified shorthand notation for determinant and $Q$-determinant}

Let $R$ be a ring and $S$ a commutative ring. Let $\tau:R\to S$ be a morphism of $\Z$-modules, that is, an additive map. Let us assume that $\tau$ is central, in the sense that for any $r,r'\in R$, we have $\tau(rr')=\tau(r'r)$. For all integers $n\geq 1$ and all $n\times n$ matrices $M\in M_{n}(R)$, we define the $\tau$-determinant of $M$ as the element of $S$ given by the formula
\begin{equation}\label{eq:deftau}
{\det}_{\tau}(M)=\sum_{\sigma\in \S_{n}} \epsilon(\sigma) \prod_{\substack{c \text{ cycle of } \sigma \\ c=(i_{1}\ldots i_{r})}} \tau(M_{i_{1}i_{2}}M_{i_{2}i_{3}}\ldots M_{i_{r}i_{1}})\,,
\end{equation}
where $\S_{n}$ denotes the symmetric group on $n$ elements.

We do not claim that the function ${\det}_{\tau}$ has particularly good properties on $M_{n}(R)$ in general. In particular, although it is $\Z$-multilinear in the columns and the rows of $M$, it is not alternating in general. However, its interest for us is that it unifies on the one hand the usual notion of determinant of a matrix with complex entries, or with entries in a ring of polynomials with complex coefficients, or indeed with entries in any commutative ring, by taking $S=R$ and $\tau={\rm id}_{R}$ ; and on the other hand the less usual notion of $Q$-determinant~\cite{Moore} (Moore's original work can only be found easily as a conference abstract, so we refer to~\cite[Section 5.1]{Mehta} for a more detailed textbook treatment) of a matrix with quaternionic entries, by taking $R=\H$, $S=\R$ and $\tau=\Re$, the real part map. It also applies to the situation where $R$ is a polynomial ring with quaternionic coefficients, and $S$ the polynomial ring with the same indeterminates and real coefficients.

\section{A matrix-tree theorem}\label{sec:MTT}

Let $n\geq 2$ be an integer and let $\V=\{1,\ldots,n\}$ be the set of vertices of a complete graph. Let $\E$ be the set of edges of this graph, that is, the set of ordered couples of distinct vertices: $\E=\{(i,j)\in \V^{2} : i\neq j\}$. We think of the edge $(i,j)$ as going from $i$ to $j$.\footnote{Since our edges will be weighted, all results can be easily re-interpreted to take into account graphs with multiple (weighted) edges. We can also allow self-loops by adding the diagonal coefficients $a_{ii}(1-h_{ii})$.}

Let $H$ be a ring, which we do not assume to be commutative. Examples that we have in mind are $\R$, $\C$, or $\H$. Let $R$ be the polynomial ring $H[a_{ij} : (i,j)\in \E]$ over $n(n-1)$ indeterminates. Let $\{h_{ij}, (i,j)\in \E\}$ be $n(n-1)$ elements of $H$. As the notation suggests, we think of $a_{ij}$ and $h_{ij}$ as being attached to the edge $(i,j)$.

We form the matrix $\Delta\in M_{n}(R)$ by setting, for all $(i,j)\in \E$,
\begin{equation}\label{eq:D1}
\Delta_{ij}=\red{-h_{ij}a_{ij}}
\end{equation}
and, for all $i\in \V$,
\begin{equation}\label{eq:D2}
\Delta_{ii}=\sum_{j\in \V \setminus\{i\}} a_{ij}\, .
\end{equation}
The use of a colour will play an instrumental role in the computation, by helping us to keep track of the origin of the coefficients.\footnote{For those reading this paper in black and white, note that the right-hand side of Eq.~\eqref{eq:D1} is red; moreover, in the rest of the text, replace the word \emph{red} by \emph{marked}. However, for those who can afford it, we recommend the coloured version of the paper.}

Let us choose a commutative ring $K$ and a central morphism of $\Z$-modules $\tau:H\to K$. Here, the examples we have in mind for $K$ are $\R$ or $\C$, and $\tau$ the identity map or the real part. Let us define $S=K[a_{ij} : (i,j)\in \E]$. The morphism $\tau$ extends to a morphism $R\to S$ that we still denote by $\tau$.

Let us fix an integer $m\in \{1,\ldots,n\}$. We consider the $m\times m$ principal submatrix $\Delta_{[m]}$ of~$\Delta$ obtained by erasing the rows and columns of indices greater than $m$. The version of the matrix-tree theorem that we are going to state computes the $\tau$-determinant of $\Delta_{[m]}$.

Let us define the \emph{boundary} of our graph, or its \emph{well}, as the subset $\W=\{m+1,\ldots,n\}\subset \V$. Let us also define $\U=\{1,\ldots,m\}=\V\setminus \W$, the set of inner vertices.

We say that a subset $\Fe\subset \E$ is a \emph{cycle-and-well-rooted spanning forest} (with well $\W$) if it contains exactly $m$ edges, one coming out of $i$ for each $i\in \U$ (see Figure~\ref{fig:wcrsf}).\footnote{These combinatorial objects or their variants are also called \emph{mappings}, \emph{pseudoforests}, or \emph{bases of the bicircular matroid} in the literature.} We denote by $\F_{m}$ the set of all cycle-and-well-rooted spanning forests with well $\W$. The arrow is meant to remind us that each element of $\F_{m}$ is a configuration of directed edges.

Consider $\Fe\in \F_{m}$. We set $a_{\Fe}=\prod_{(i,j)\in \Fe}a_{ij}$. Moreover, using the decomposition into connected components of the undirected graph underlying $(\V,\Fe)$, we can partition the directed graph $(\V,\Fe)$ into disjoint pieces, each of which is of one of two kinds: either a tree rooted at a vertex of $\W$ with all edges consistently pointing towards the root, or a cycle-rooted tree (also called a unicycle) contained in $\U$. To $\Fe$, we attach the collection $\Cyc(\Fe)$ of the cycles of its unicycles. It is a (possibly empty) family of pairwise disjoint simple cycles in $\U$. For each such cycle $c$, visiting the vertices $i_{1},\ldots,i_{r}$ in this cyclic order, we can form the ill-defined element $h_{c}=h_{i_{1}i_{2}}\ldots h_{i_{r}i_{1}}$ of $H$, and the well-defined element $\tau(h_{c})=\tau(h_{i_{1}i_{2}}\ldots h_{i_{r}i_{1}})$ of $K$.

\begin{theorem}\label{thm:MTKZ}
In the ring $S=K[a_{ij},(i,j)\in \E]$, we have the equality
\[{\det}_{\tau} \Delta_{[m]}=\sum_{\Fe\in \F_{m}} a_{\Fe} \prod_{c\in \Cyc(\Fe)} \big(1-\tau(h_{c})\big)\,.\]
\end{theorem}

\begin{proof} Let us write the definition \eqref{eq:deftau} of the $\tau$-determinant of $\Delta_{[m]}$ and separate, in the contribution of each permutation, the fixed points from the cycles of length at least $2$. We have
\begin{align}
\nonumber
{\det}_{\tau}\Delta_{[m]}&=\sum_{\sigma \in \S_{m}} \epsilon(\sigma) \prod_{c=(i_{1}\ldots i_{r})} \tau(\Delta_{i_{1}i_{2}}\ldots \Delta_{i_{r}i_{1}})\\
&=\sum_{\sigma \in \S_{m}} \epsilon(\sigma) \prod_{i:\, \sigma(i)=i} \tau(\Delta_{ii}) \prod_{c=(i_{1}\ldots i_{r}),\, r\geq 2} \tau(\Delta_{i_{1}i_{2}}\ldots \Delta_{i_{r}i_{1}})\, .
\label{eq:somme}
\end{align}

For each permutation $\sigma\in \S_{m}$, we will interpret the term of this sum indexed by $\sigma$ as a sum of a certain number of contributions, each being attached to a certain configuration of black and red\footnote{Or \emph{marked}.} edges on our graph. There will be, for each permutation $\sigma$, a number of edge configurations equal to $(n-1)^{f}$, where $f$ is the number of fixed points of $\sigma$. Each edge configuration will be a cycle-and-well-rooted spanning forest, in which some of the cycles are made of red edges. 

To be more precise, let us consider a particular permutation $\sigma$. For each cycle $c=(i_{1}\ldots i_{r})$ of $\sigma$ with length at least $2$, we have
\[\tau(\Delta_{i_{1}i_{2}}\ldots \Delta_{i_{r}i_{1}})=(-1)^{r} \tau(\red{h_{i_{1}i_{2}}\ldots h_{i_{r}i_{1}}}) \red{a_{i_{1}i_{2}}\ldots a_{i_{r}i_{1}}}\,.\]
We think of this term as associated with a cycle of red edges on our graph, namely the cycle formed by the edges $\red{(i_{1},i_{2})},\ldots,\red{(i_{r},i_{1})}$.

Then, for each fixed point $i$ of $\sigma$, we have 
\[\tau(\Delta_{ii})=\tau(a_{i1}+\ldots + \widehat{a_{ii}}+\ldots + a_{in})=a_{i1}+\ldots + \widehat{a_{ii}}+\ldots + a_{in}\,.\]
We think of these $n-1$ terms as being respectively associated with the $n-1$ possible ways of drawing one edge in our graph coming out of the vertex $i$. This edge, as the notation suggests, is a black edge.

Expanding the term indexed by $\sigma$ in \eqref{eq:somme} over the sum produced by each fixed point of $\sigma$ makes it appear as the sum of $(n-1)^{f}$ contributions, where $f$ is the number of fixed points of $\sigma$, each contribution being attached to a configuration of black and red edges on our graph. Each edge configuration contains exactly one edge coming out of each vertex of $\U=\{1,\ldots,m\}$ and is thus, in the terminology introduced above, a cycle-and-well-rooted spanning forest. Moreover, in each cycle-and-well-rooted spanning forest that appears, some of the cycles are made of red edges. Here the phrase {\em some of the cycles} is a shorthand for {\em an arbitrary, possibly empty and possibly full, subset of the set of cycles}. 

Let us call {\em cycle-coloured cycle-and-well-rooted spanning forest} a configuration of black and red edges on the graph of the kind that we just described, and illustrated in Figure \ref{fig:wcrsf} below. 

\begin{figure}[h!]
\begin{center}
\includegraphics{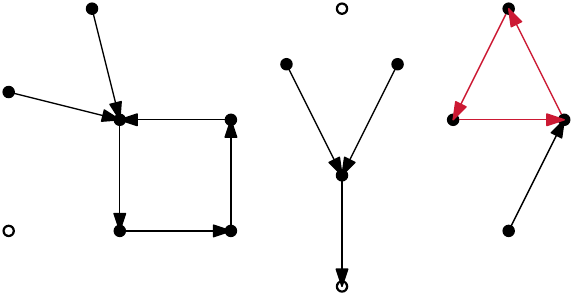}
\caption{\label{fig:wcrsf} A cycle-coloured cycle-and-well-rooted spanning forest. The vertices belonging to the well are white.}
\end{center}
\end{figure}

We claim that each cycle-coloured cycle-and-well-rooted spanning forest appears once and only once in the whole sum \eqref{eq:somme}. Indeed, consider such an edge configuration. Form the permutation~$\sigma$ of $\{1,\ldots,m\}$ of which the cycles of length at least $2$ are exactly the red cycles of the edge configuration. The fixed points of $\sigma$ are thus the vertices of $\U$ out of which comes a black edge. Then this particular edge configuration appears exactly once in the term of \eqref{eq:somme} indexed by that permutation $\sigma$, and in no other.

Finally, the contribution of a cycle-coloured cycle-and-well-rooted spanning forest to \eqref{eq:somme} is the product of 
\begin{itemize}
\item the indeterminates $a_{ij}$ attached to the edges of the configuration, regardless of their colour,
\item a term $-\tau(\red{h_{c}})$ for each red cycle $c$ of the configuration.
\end{itemize}
The minus sign which appears in the contribution of a red cycle of length $r$ is the product of a sign $(-1)^{r}$ coming from the fact that the off-diagonal entries of $\Delta$ carry a minus sign, and the contribution $(-1)^{r-1}$ of this cycle to the signature of $\sigma$.

At this point, we have replaced the sum \eqref{eq:somme} by a larger sum indexed by (black and red) cycle-coloured cycle-and-well-rooted spanning forests. We are now going to reorganise the terms of this larger sum and make it into a sum over (black\footnote{That is, \emph{uncoloured}, or \emph{unmarked}.}) cycle-and-well-rooted spanning forests. Indeed, let us consider a cycle-and-well-rooted spanning forest $\Fe\in \F_{m}$. We will assign to $\Fe$ the sum of all the contributions coming from the various ways of colouring none, some, or all of the cycles of $\Fe$ in red. There are $2^{b}$ contributions, where $b$ is the number of cycles of $\Fe$, and their sum is equal to the product of 
\begin{itemize}
\item the indeterminates $a_{ij}$ attached to the edges of $\Fe$, that is, $a_{\Fe}$,
\item a term $1-\tau(\red{h_{c}})$ for each cycle $c$ of $\Fe$, accounting for the fact that this cycle can be left black or coloured in red.
\end{itemize}
This is exactly the announced formula.
\end{proof}

In many applications, the indeterminates $\{a_{ij}, (i,j)\in \E\}$ are specialised in a symmetric way, that is, in such a way that the equality $a_{ij}=a_{ji}$ is satisfied for each edge $(i,j)\in \E$. Theorem \ref{thm:MTKZ} can be nicely adapted to this situation. Indeed, in this symmetric case, the monomial $a_{\Fe}$ attached to a cycle-and-well-rooted spanning forest $\Fe$ is not affected by reversing the orientation of one or several cycles of $\Fe$. To acknowledge this invariance, let us declare equivalent two elements of~$\F_{m}$ which differ only by the orientation of some of their cycles, and let us denote by $\UUF_{m}$ the quotient set. 

Let $[\Fe]$ be the class in $\UUF_{m}$ of a cycle-and-well-rooted spanning forest $\Fe$. The obvious statement that the elements of $[\Fe]$ are deduced from $\Fe$ by changing the orientation of some of the cycles of $\Fe$ comes with the nice twist that the two orientations of a cycle of length $2$ are in fact one and the same. For this reason, if $c$ is a cycle of $\Fe$ of length $2$, then $c=c^{-1}$, hence $\tau(h_{c})=\tau(h_{c^{-1}})$ is well defined and depends only on $[\Fe]$. On the other hand, if $c$ is a cycle of $\Fe$ of length at least $3$, then it is the pair $\{\tau(h_{c}),\tau(h_{c^{-1}})\}$ that depends only on $[\Fe]$. 

Summing the expression of ${\det}_{\tau}\Delta_{[m]}$ provided by Theorem \ref{thm:MTKZ} over equivalence classes yields the following result. We denote the length of a cycle $c$ by $\ell(c)$.

\begin{corollary}\label{cor:sym} In the quotient ring $S/(a_{ij}-a_{ji}:(i,j)\in \E)$, we have the equality
\[{\det}_{\tau} \Delta_{[m]}=\sum_{[\Fe]\in \UUF_{m}} a_{\Fe} \prod_{\substack{c\in \Cyc(\Fe)\\ \ell(c)=2}} \big(1-\tau(h_{c})\big) \prod_{\substack{c\in \Cyc(\Fe)\\ \ell(c)\geq 3}} \big(2-\tau(h_{c})-\tau(h_{c^{-1}})\big)\,.\]
\end{corollary}

\section{Special cases}

\subsection{The classical matrix-tree theorem}

Consider $H=K=\R$ and let $\tau$ be the identity map. Take $h_{ij}=1$ for each edge $(i,j)$. Then in the sum in Theorem \ref{thm:MTKZ}, non-zero contributions arise only from cycle-and-well-rooted spanning forests without any cycle, that is, from spanning forests rooted in the well. We recover exactly the classical matrix-tree theorem.

\subsection{Forman's theorem}

Consider $H=K=\C$ and let $\tau$ be the identity map. Take $h_{ij}$ to be any complex number for each edge $(i,j)$. Then we recover Forman's matrix-tree theorem with holonomies\footnote{We call {\em holonomies} the elements $h_{ij}$ and their products along paths in the graph. This terminology is inspired by the setup where the graph is embedded in a manifold and where $h_{ij}$ is the parallel transport of a connection along edge $ij$.} along the edges that are complex numbers (\cite[Theorem~1]{Forman}). If the variables $a_{ij}$ are symmetric, and if the holonomies are taken to be complex numbers such that $h_{ji}=\overline{h_{ij}}$ for each edge $(i,j)$, then we can apply Corollary \ref{cor:sym} to find that cycles of length $2$ contribute by a factor $1-\Re(h_{c})$ and cycles of length at least $3$ by a factor $2-2\Re(h_{c})$. Assuming moreover that the complex numbers~$h_{ij}$ have modulus $1$ forbids the presence of cycles of length $2$ in any of the cycle-and-well rooted forests in the sum, since in this case, $h_{c}=1$ for any such cycle.

\subsection{Chaiken's theorem}

Consider $H=K=\Z[\Gamma]$, where $\Gamma$ is an abelian group, and $\tau$ the identity map. Then Theorem~\ref{thm:MTKZ} is the formula in the middle of p.326 of~\cite{Chaiken} in the case of principal minors. By further specialising $\Gamma$ to be the multiplicative group $\{\pm 1\}$, we recover Zaslavsky's theorem about signed graphs (\cite[Theorem~8A.4]{Zaslavsky}), also a special case of Forman's theorem.

\subsection{Kenyon's theorem}

Consider $H=\H$, $K=\R$ and let $\tau$ be the real part map. Choose the quaternions $h_{ij}$ to be of norm $1$ for each edge $(i,j)$, and such that $h_{ji}=\overline{h_{ij}}$. Then, up to the identification of the unit sphere of $\H$ with $\SU(2)$, Corollary \ref{cor:sym} reduces to Kenyon's theorem with holonomies in $\SU(2)$ (\cite[Theorem~9]{Kenyon}, which Kenyon stated in the case of an empty well only, although his proof extends readily to the case of a non-empty well).

\section{Graphs with holonomy in a space of matrices}\label{sec:MTTN}

In our research about stochastic processes on graphs with gauge symmetry~\cite{Kassel-ESAIM, KL1, KL-QSF}, our main motivation was to understand the determinant of $\Delta_{[m]}$ when the elements $h_{ij}$ are matrices, for instance unitary matrices. This is an ill-posed question in the sense that several expressions, with various merits and demerits, can be given for this determinant. 

For example, if we are interested in $N\times N$ complex matrices, we can apply Theorem \ref{thm:MTKZ} with $H=M_{N}(\C)$, $K=\C$ and $\tau(\cdot)=\frac{1}{N}\Tr(\cdot)$. One difficulty with the formula that we obtain in this way is that it seems difficult to make sense of the definition ${\det}_{\tau}\Delta_{[m]}$. We propose an alternative approach, in which we imitate and adapt the proof of Theorem \ref{thm:MTKZ}.

Let us describe the situation again. We have a ring $H$, not assumed to be commutative, a commutative ring $K$ and a central morphism $\tau:H\to K$.
We have the set of vertices $\V=\{1,\ldots,n\}$,  in which a certain subset $\W=\{m+1,\ldots,n\}$ is called the well. To each edge $(i,j)\in \E$ is attached an indeterminate $a_{ij}$ and an $N\times N$ matrix $h_{ij}\in M_N(H)$. We form the matrix $\Delta$ following \eqref{eq:D1} and \eqref{eq:D2}, with $a_{ij}$ understood as $a_{ij}I_{N}$ in \eqref{eq:D2}. It is an $n\times n$ matrix of $N\times N$ matrices, which we see in the most natural way as an $Nn\times Nn$ matrix with entries in~$H$. 

We will compute the $\tau$-determinant of the $Nm\times Nm$ matrix $\Delta_{[Nm]}$ obtained from $\Delta$ by erasing the rows and columns of indices greater than $Nm$.

For this, we introduce, above our usual graph $\G=(\V,\E)$, a new graph which we denote by~$\GbN$. The set of vertices of this new graph is $\V\times \{1,\ldots,N\}$. Thus, a vertex of $\GbN$ is a pair $(i,k)$ where $i\in \{1,\ldots,n\}$ is a vertex of $\G$ and $k\in \{1,\ldots,N\}$ can be used to index a row or a column of an $N\times N$ matrix. The edges of $\GbN$ are all pairs $((i,k),(j,l))$ such that $(i,j)$ is an edge of $\G$ and $(k,l)\in \{1,\ldots,N\}^{2}$. It is an important feature of the graph $\GbN$ that it contains no `vertical' edge (\emph{i.e.} an edge of the form $((i,k),(j,l))$ with $i=j$).\footnote{We are not aware of a standard name for this simple construction, but as pointed out to us by the referee, a more combinatorially idiomatic characterisation is that $\GbN$ is the lexicographic product of $\G$ with the edgeless graph on $N$ vertices.}

\begin{figure}[h!]
\begin{center}
\includegraphics{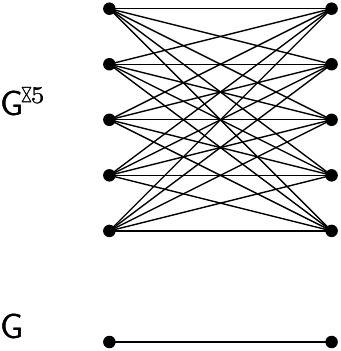}
\caption{\label{fig:bt} \small The new graph
has $N^{2}$ directed edges over each directed edge of the old graph.}
\end{center}
\end{figure}

To each edge $((i,k),(j,l))$ of $\GbN$, we associate on the one hand the indeterminate $a_{ij}$ and on the other hand the element $\hb_{((i,k),(j,l))}=(h_{ij})_{kl}$ of $H$, that is, the $(k,l)$-th entry of the matrix~$h_{ij}$.  

We define the well of the graph $\GbN$ as the set $\W\times \{1,\ldots,N\}$ of all vertices located `above' a vertex of $\W$. Our formula involves a particular subset of the set $\F_{m,N}$ of cycle-and-well-rooted spanning forests in $\GbN$. To describe this set, let us define a horizontal edge as an edge $((i,k),(j,l))$ such that $k=l$. Then, we say that an element $\Fe$ of $\F_{m,N}$ is \emph{horizontal} if every edge of $\Fe$ that is not in a cycle is horizontal. We denote by $\HF_{m,N}$ the corresponding subset of~$\F_{m,N}$. Moreover, we say that a cycle of $\Fe$ is horizontal if it consists only of horizontal edges. We say that a cycle is \emph{skew} if it is not horizontal. We denote by $\Hcyc(\Fe)$ the set of horizontal cycles of $\Fe$, and by $\Scyc(\Fe)$ the set of its skew cycles. For every cycle $c$ of $\Fe$, we define $\hb_{c}$ as the product of the complex numbers $\hb_{((i,k),(j,l))}$ associated with the edges which constitute $c$.

Our result is the following.

\begin{theorem}\label{thm:MTKZN}
In the ring $S=K[a_{ij},(i,j)\in \E]$, we have the equality
\[{\det}_{\tau}  \Delta_{[Nm]}=\sum_{\Fe\in \HF_{m,N}} a_{\Fe} \prod_{c\in \Hcyc(\Fe)} \big(1-\tau(\hb_{c})\big) \prod_{c\in \Scyc(\Fe)}\big(-\tau(\hb_{c})\big)\,.\]
\end{theorem}

\begin{proof} The proof is very similar to the proof of Theorem \ref{thm:MTKZ}. The main difference is that we see the $Nm$ indices of the rows and of the columns of $\Delta_{[Nm]}$ as corresponding to the $Nm$ vertices of the new graph $\GbN$: the vertex $(i,k)$ corresponds to the $(N(i-1)+k)$-th row, or column.

We expand the determinant of $\Delta_{[Nm]}$ as a sum indexed by permutations, which we really think of as permutations of the set $\{1,\ldots,m\}\times \{1,\ldots,N\}$.
We separate, for each permutation, the  fixed points from the non-trivial cycles. 

The contribution of each non-trivial cycle is understood as being associated with a cycle of red edges on the graph $\GbN$. This is possible because the diagonal $N\times N$ blocks of $\Delta$ are diagonal $n\times n$ matrices (recall the third paragraph of the current Section~\ref{sec:MTTN}), so that the contribution of any permutation containing a cycle which would produce a `vertical' edge vanishes. To be clear, the contribution of a permutation $\sigma$ such that there exists $i\in \{1,\ldots,m\}$ and $k,l$ distinct elements of $\{1,\ldots,N\}$ such that $\sigma(i,k)=(i,l)$ is zero, because $\Delta_{(i,k),(i,l)}=0$.

The fixed points contribute by producing a sum over all ways of completing the configuration of red cycles that we just obtained, by adding a horizontal black edge coming out of every vertex of $\{1,\ldots,m\}\times \{1,\ldots,N\}$ that is not already visited by a red cycle. 

Let us emphasize that there is something arbitrary in our decision of assigning the contributions of fixed points to {\em horizontal} edges (we come back to this point in Section~\ref{sec:horizontal}). Let us at least spell out what this exactly means: if $(i,k)$ is a fixed point of $\sigma$, then in expanding the sum $a_{i1}+\ldots + \widehat{a_{ii}}+\ldots + a_{in}$, we will attach, for each $j\neq i$, the weight $a_{ij}$ to the edge $((i,k),(j,k))$.

The outcome of this procedure is an expression of the determinant of $\Delta_{[Nm]}$ as a sum over cycle-and-well-rooted spanning forests in $\GbN$, with well $\{m+1,\ldots,n\}\times \{1,\ldots,N\}$. Thanks to our decision of assigning the contributions of fixed points to horizontal edges, all the forests which appear are horizontal, in the sense that every edge that is not in a cycle is horizontal. Moreover, each of these forests has  a certain number of its cycles made of red edges. To complete our description, we must say that any cycle that is not horizontal must be made of red edges, because all black edges are horizontal. 

Just as in the proof of Theorem \ref{thm:MTKZ}, we can check that every horizontal cycle-and-well-rooted spanning forest in $\GbN$ with all its skew cycles and an arbitrary subset of the set of its horizontal cycles made of red edges contributes exactly once to the sum, with a contribution equal to the product of 
\begin{itemize}
\item the indeterminates $a_{ij}$ attached to the edges of the configuration, regardless of their colour,
\item a term $-\tau(\red{\hb_{c}})$ for each red cycle $c$.
\end{itemize}
Reorganising the sum by summing first over sets of configurations which differ only by the colour of their edges, and taking into account the fact that skew cycles are always red, whereas horizontal cycles can be either black or red, we obtain the announced formula.
\end{proof}

In the case where the indeterminates $\{a_{ij},(i,j)\in \E\}$ are symmetric, the discussion of the end of Section~\ref{sec:MTT} applies and, by summing over equivalence classes of cycle-and-well-rooted spanning forests and using self-explanatory notation, we obtain the following result.

\begin{corollary}\label{cor:symN} In the quotient ring $S/(a_{ij}-a_{ji}:(i,j)\in \E)$, we have the equality
\begin{align*}
{\det}_{\tau}  \Delta_{[Nm]}=&\sum_{[\Fe]\in \mathscr{H}\!\UUF_{m,N}} a_{\Fe} \prod_{\substack{c\in \Hcyc(\Fe)\\ \ell(c)=2}} \big(1-\tau(h_{c})\big) \prod_{\substack{c\in \Hcyc(\Fe)\\ \ell(c)\geq 3}} \big(2-\tau(h_{c})-\tau(h_{c^{-1}})\big)\\
&\hspace{4cm}\prod_{\substack{c\in \Scyc(\Fe)\\ \ell(c)=2}} \big(-\tau(h_{c})\big) \prod_{\substack{c\in \Scyc(\Fe)\\ \ell(c)\geq 3}} \big(-\tau(h_{c})-\tau(h_{c^{-1}})\big)
\,.
\end{align*}
\end{corollary}

Let us stress that, just as in Theorem~\ref{thm:MTKZN}, cancellations may occur in the right-hand side of the equation of Corollary~\ref{cor:symN}. When $a_{ij}=a_{ji}$, we can easily factor $\Delta$ as a product of rectangular matrices. Provided we are in a case where it holds for the $\tau$-$\det$\footnote{\label{footnote}This is the case when $H$ is commutative and $\tau=\mathrm{id}$ or when $H=\mathbb{H}$, $h$ is unitary, and $\tau=\Re$.}, the Cauchy--Binet formula implies that the coefficients of the monomials $a_{\Fe}$ such that the degree of $a_{ij}$ exceeds $N$ are zero. However, we do not know a direct proof of this fact.

\section{Applications}

\subsection{The case $N=1$}

We recover Theorem~\ref{thm:MTKZ} and Corollary~\ref{cor:sym}.

\subsection{The unitary case}

When $H=\R,\C$ or $\H$ and the matrices $h_{ij}$ are unitary, we have for any cycle $c\in\GbN$ that $\tau(h_c)$ and $\tau(h_{c^{-1}})$ are complex conjugates so that we can slightly simplify the equation in Corollary~\ref{cor:symN} using real parts, writing $2\Re(\tau(h_c))$ in place of $\tau(h_c)+\tau(h_{c^{-1}})$. Moreover, in that case, the Cauchy--Binet formula alluded to right above shows that the coefficient of each monomial $a_{\Fe}$ is nonnegative.

\subsection{The case of block-diagonal matrices}

If the matrices $h_{ij}$ are block-diagonal in a consistent way, then the determinant factorises, a fact that is clear from the point of view of the left-hand side of the formula.

\subsection{Simplicial complexes}

Analogs of the classical matrix-tree theorem have been proven for simplicial complexes, starting with Kalai~\cite{Kalai} and followed by works of Adin~\cite{Adin}, Lyons~\cite{Lyons} and Duval--Klivans--Martin~\cite{DKM} among others.

Recall that a simplicial complex $\bC$ on the vertex set $\bC_0=\{1,\ldots, v\}$ is a collection of non-empty subsets of $\bC_0$, called cells, which contains the singletons $\{1\},\ldots,\{v\}$ and which contains every non-empty subset of every set that it contains.

For concreteness we consider here the simplicial complex $\bC$ of all non-empty subsets of $\bC_0$ of cardinality less than or equal to $d+1$ for some fixed $d\in\{1,\ldots,v-1\}$. For any $k\in \{0,\ldots,d\}$, it contains ${v \choose k+1}$ subsets of cardinality $k+1$, which we call cells of dimension $k$, or $k$-cells.

A graph is a simplicial complex of dimension $1$, and its Laplacian is an operator acting on $0$-chains. In higher dimension, the analogs of the matrix-tree theorem concern a natural operator on $(d-1)$-chains, the so-called `top-down' Laplacian (coming from the structure of chain complex with boundary operator $\partial$ and coboundary operator $\delta$ and defined as $\Delta=\partial\delta$). Here we consider a version of this operator twisted by elements of a ring $H$. For the purposes of this short paper, and in tune with Zeilberger's opening lines of~\cite{Zeilberger}, we do not detail the definition of this operator from the algebraic-topological point of view of the chain complex, but rather directly define it combinatorially as a matrix indexed by the set of $(d-1)$-cells. This matrix depends on a choice of orientation of the cells, but in a way that does not affect its principal minors.

Let $\bC_{d-1}$ be the set of all $(d-1)$-cells of $\bC$. We call two elements $\sigma,\tau$ of $\bC_{d-1}$ {\em adjacent}, and we write $\sigma\sim\tau$, if $\sigma\cup\tau$ is a $d$-cell. We denote by $\E$ the set of ordered pairs of adjacent $(d-1)$-cells in $\bC$. To each $(\sigma,\tau)\in\E$ we assign an indeterminate $a_{\sigma\tau}$ and an element $h_{\sigma\tau}$ of $H$.

We fix an arbitrary choice of orientation for each $(d-1)$-cell. This choice determines for each pair $\sigma \sim \tau$ a sign $\epsilon_{\sigma\tau}$, the exact definition of which need not worry us here. The important fact is that for any chain $c=(\sigma_1\ldots\sigma_k)$ of adjacent $(d-1)$-cells with $\sigma_k\sim\sigma_1$, the product of the signs $\varepsilon_{\sigma_j\sigma_{j+1}}$ when $j$ runs cyclically through $\{1,\ldots, k\}$, denoted by $\varepsilon_c$, is independent of the choice of orientation of the cells $\sigma_j$.\footnote{As pointed out to us by the referee, the above combinatorial setup can be extended from simplicial complexes to finite CW complexes by allowing the signs $\varepsilon_{\sigma\tau}$ to become arbitrary integers. See~\cite{DKM-survey} for a survey on spanning trees in the cellular case.}

We can now write the non-zero entries of the `matrix'\footnote{To be precise, we did not specify an ordering of the indices so it is only a matrix up to permutation, a fact which bears no consequence on the computation of its principal minors.} $\dDelta$ by setting, for all $(\sigma,\tau)\in\E$,
\begin{equation}\label{eq:C1}\dDelta_{\sigma\tau}=\red{-\varepsilon_{\sigma\tau}h_{\sigma\tau}a_{\sigma\tau}}
\end{equation} 
and for all $\sigma\in\bC_{d-1}$, 
\begin{equation}\label{eq:C2}
\dDelta_{\sigma\sigma}=\frac{1}{d}\sum_{\substack{\tau\in \bC_{d-1}\setminus \{\sigma\}\\ \tau\sim \sigma}}a_{\sigma\tau}\,.
\end{equation}

This operator $\Delta$ can indeed be written in the form $\partial \delta$ whenever the weights satisfy some symmetry conditions, namely: for $\sigma\sim\tau$, the weight $a_{\sigma\tau}$ is only a function of the $d$-cell $\rho=\sigma\cup \tau$ and moreover there exist weights $h_{\sigma\rho}$ and $h_{\rho\tau}$ such that $h_{\sigma\tau}=h_{\sigma\rho}h_{\rho\tau}$.

The similarity between Eqs.~\eqref{eq:C1} \& \eqref{eq:C2} and Eqs.~\eqref{eq:D1} \& \eqref{eq:D2} suggests to look at $\dDelta$ as a Laplacian matrix on the weighted graph $\tilde{\G}=(\bC_{d-1},\E)$ with weights $\tilde{a}_{\sigma\tau}=a_{\sigma\tau}/d$ and $\tilde{h}_{\sigma\tau}=d\, \varepsilon_{\sigma\tau}h_{\sigma\tau}$, when $\sigma\sim\tau$; and $\tilde{a}_{\sigma\tau}=0$ and $\tilde{h}_{\sigma\tau}=1$, otherwise.

We are now exactly in the setup of Section~\ref{sec:MTTN} and we can thus use our Theorem~\ref{thm:MTKZN} to compute the principal minors of $\dDelta$. For simplicity, we only state the $N=1$ case, which is thus a corollary of Theorem~\ref{thm:MTKZ}. Fix $m\in\{1,\ldots, {v\choose d}\}$.

\begin{theorem}\label{thm:CW}
In the ring $S=K[a_{ij},(i,j)\in \E]$, we have the equality
\[{\det}_{\tau} \dDelta_{[m]}=\frac{1}{d^{m}}\sum_{\Fe\in \F_{m}} a_{\Fe} \prod_{c\in \Cyc(\Fe)} (1-d^{\ell(c)}\varepsilon_c\tau(h_{c}))\,.\]
\end{theorem}

As in the discussion at the end of Section~\ref{sec:MTTN}, an application of the Cauchy--Binet formula (when it holds, see Footnote~\ref{footnote}) in the case where the weights are symmetric enough that $\dDelta=\partial\delta$ (which means in particular that for all $\sigma\sim\tau$, the weight $a_{\sigma\tau}$ is a function $x_\rho$ of the $d$-cell $\rho=\sigma\cup\tau$) shows the cancellation of many terms in the sum: all terms for which the monomials $a_{\Fe}$ contain $x_\rho$ with degree exceeding~$1$.\footnote{In that special case, the equation of Theorem~\ref{thm:CW} holds with the sum restricted to a certain subset of $\F_{m}$.} Moreover the operator $\dDelta$ then has a non-zero kernel (the image of the coboundary operator $\delta$ on $(d-2)$-cells) so all minors cancel for $m$ large enough. In particular, choosing $h=1$ and $m={v\choose d}$, the determinant is zero, something which is not clearly apparent from the right-hand side of the formula.

The case $h=1$ has been studied considerably, and in that case, the coefficient of each monomial $x_{\Fe}$ in the maximal non-zero principal minors of $\Delta$ is the square of the torsion of a certain $(d-1)$-subcomplex called a simplicial spanning tree.\footnote{By torsion, we mean the cardinality of a certain torsion subgroup of the homology of that subcomplex.} Our formula expresses this torsion in terms of a weighted count of cycles. This can be compared to the formulas of~\cite{Bernardi-Klivans} expressing the torsion in terms of related combinatorial sums involving cancellations. 

Are there direct proofs of these identities obtained by comparing two polynomial expansions? In particular, can one compare our Theorem~6.1 with Theorem~31 and Remark~32 of~\cite{Bernardi-Klivans} and go from one statement to the other? Our statement is yet another evidence that homological quantities may be expressed in geometrico-combinatorial terms as shown by Bernardi and Klivans. We leave these open questions to the interested reader.

\section{Loose ends}

\subsection{Horizontal edges}\label{sec:horizontal}

As noted in the course of the proof of Theorem~\ref{thm:MTKZN}, there was something arbitrary about interpreting the black edges as horizontal.

By averaging, in the proof of Theorem~\ref{thm:MTKZN}, over all possible inclinations that we can assign to black edges (not merely the \emph{horizontal} one), or equivalently by applying Theorem \ref{thm:MTKZ} to the graph $\GbN$, we obtain a variant of Theorem~\ref{thm:MTKZN}. The main difference is that we now sum over \emph{all} cycle-and-well-rooted spanning forests of $\GbN$. There are many more terms in the sum (which thus cancel out somehow), but the formula is slightly slicker. We leave the easy details to the interested readers.

\begin{theorem}\label{thm:MTTNall}
With the notation of Section \ref{sec:MTTN}, we have, in the ring $S=K[a_{ij},(i,j)\in \E]$, the equality
\[{\det}_{\tau}  \Delta_{[Nm]}=\frac{1}{N^{Nm}}\sum_{\Fe\in \F_{m,N}} a_{\Fe} \prod_{c\in \Cyc(\Fe)}\big(1-N^{\ell(c)}\tau(\hb_{c})\big)\,.\]
\end{theorem}

We note the close similarity between the identities in Theorem~\ref{thm:MTTNall} and Theorem~\ref{thm:CW} about simplicial complexes. This suggests a maybe not so surprising combinatorial analogy between the study of vector bundles of rank $1$ over simplicial complexes of dimension $d$ with $v=d n$ vertices, and the study of vector bundles of rank $N=d$ over graphs with $n$ vertices.

This similarity also suggests that there could exist a formula for simplicial complexes of the more concise form of Theorem~\ref{thm:MTKZN}.

\subsection{Gauge invariance}

One shortcoming of our Theorem~\ref{thm:MTKZN} is that the terms in the sum of the right-hand side are not manifestly gauge-invariant, nor non-negative when holonomies are unitary. We propose solutions to this problem in an upcoming work.

\subsection{Non-principal minors}

One question we have not adressed in this note is the computation of non-principal minors of discrete Laplacians. In the classical case, it is known that these minors also enumerate spanning trees. 

Chaiken~\cite[p. 326]{Chaiken} computed a formula for non-principal minors of so-called `gain-graphs/voltage-graphs' which in our language are graphs with holonomy taking values in an abelian group.

It is possible to work out formulas in our case too, and we leave this endeavour to the curious reader.

\bigskip

{\small \textsc{Acknowledgments.} We thank David Wilson for teaching us Zeilberger's proof in an enlightening way using coloured chalks on a blackboard of the Newton Institute in Cambridge during the Random Geometry programme in 2015. Unknowingly, it planted the seed for the above proofs. We also thank Omid Amini and Gr\'egory Miermont for stimulating and inspiring discussions which encouraged us to write this note. We are also grateful to Omid Amini for his help in checking Theorem \ref{thm:CW}. Finally, we thank the anonymous referee for a careful reading and helpful feedback.}


\def\@rst #1 #2other{#1}
\renewcommand\MR[1]{\relax\ifhmode\unskip\spacefactor3000 \space\fi \MRhref{\expandafter\@rst #1 other}{#1}}
\renewcommand{\MRhref}[2]{\href{http://www.ams.org/mathscinet-getitem?mr=#1}{MR#1}}
\newcommand{\arXiv}[1]{\href{http://arxiv.org/abs/#1}{arXiv:#1}}

\bibliographystyle{hmralphaabbrv}
\bibliography{colourmatrixtree}

\end{document}